\documentclass[11pt,reqno]{amsart}
\usepackage{graphicx}
\usepackage{float}
\usepackage[caption = false]{subfig}
\usepackage{amsmath}
\usepackage{enumerate}
\usepackage{mathtools}
\usepackage[english]{babel}
\usepackage{amsfonts,amssymb}
\usepackage{mathrsfs}
\usepackage[utf8x]{inputenc}\usepackage[english]{babel}
\usepackage{soul}
\usepackage[all]{xy,xypic}
\usepackage{cite}
\usepackage{relsize}
\numberwithin{equation}{section}
\newtheorem{theorem}{Theorem}[section]
\newtheorem{corollary}[theorem]{Corollary}
\newtheorem{lemma}[theorem]{Lemma}

\theoremstyle{definition}

\theoremstyle{remark}

\numberwithin{equation}{section}
\DeclareMathOperator{\RE}{Re}

\begin{document}
	
	\title[\tiny{Sharp bounds of third Hankel determinant}]{Sharp bounds of third Hankel determinant for a class of starlike functions and a subclass of $q$-starlike functions}

		\author[S. Banga]{Shagun Banga}
	\address{Department of Applied Mathematics, Delhi Technological University, Delhi--110042, India}
	\email{shagun05banga@gmail.com}

	\author[S.S. Kumar]{S. Sivaprasad Kumar$^{*}$}
	\address{Department of Applied Mathematics, Delhi Technological University, Delhi--110042, India}
	\email{spkumar@dce.ac.in}

	\thanks{$^{*}$ corresponding author}
	\subjclass[2010]{30C45, 30C50}
	
	\keywords{Sharp bounds, Third Hankel determinant, $q$-starlike functions, Carath\'{e}odory coefficients, lemniscate of Bernoulli.}
	\maketitle
	\begin{abstract}
		Following the trend of coefficient bound problems in Geometric Function Theory, in the present paper, we obtain the sharp bound of $|H_3(1)|$ for the class $\mathcal{S}^*$, of starlike functions and $\mathcal{SL}_q^*$, of $q$- starlike functions related with lemniscate of Bernoulli. Bound on the initial class is also an improvement over the existing known bound and the bound on the latter class generalizes the prior known outcome. Further, we determine the extremal functions to prove the sharpness of our results. 
	\end{abstract}
	\maketitle
	
	\section{Introduction}
	\label{intro}
	Denote the class of analytic functions $f(z)=z+\sum_{n=2}^{\infty}a_n z^n$, defined on the open unit disk $\mathbb{D}$ by $\mathcal{A}$. Let $\mathcal{S}$ be the subclass of $\mathcal{A}$ consisting of the univalent functions. For two analytic functions $f$ and $g$, we say $f$ is subordinate to $g$ if there exists a Schwarz function $\omega(z)$ with $\omega(0)=0$ and $|\omega(z)| <1$ such that $f(z) = g(\omega(z))$. 
The normalized function $f$ in $\mathcal{S}$ satisfying the following inequality
$$ \RE \dfrac{z f'(z)}{f(z)} >0, \quad z \in \mathbb{D},$$ belongs to the class of starlike functions, denoted by $\mathcal{S}^*$.  Furthermore, various sub-classes of $\mathcal{S}^*$ have been introduced and studied by many authors in the past (see~\cite{jan,men,robert,ron}). Likewise, Sok\'{o}\l~and Stankiewicz~\cite{Sok} introduced the class $ \mathcal{SL^*}$, defined as
\begin{align*}\label{sl} \mathcal{SL^*} :=\left\{f \in \mathcal{A}: \dfrac{z f'(z)}{f(z)} \prec \sqrt{1+z}, \quad z \in \mathbb{D}\right\}.\end{align*} Since then, enormous work is done for the class $\mathcal{SL^*}$, for ready reference (see~\cite{ali2012,aliravi2012,kumar2013,shelly2015,sokol22009}). 

Let $\tilde{q}$, $n \in \mathbb{N}$. For a function $f \in \mathcal{A}$, the $\tilde{q}^{th}$ Hankel determinant, is defined as follows:
\begin{equation*}\label{hqn}
	H_{\tilde{q}}(n) :=
	\begin{vmatrix}
		a_n & a_{n+1} & \ldots & a_{n+\tilde{q}-1}\\
		a_{n+1} & a_{n+2} & \ldots & a_{n+\tilde{q}}\\
		\vdots & \vdots & \ddots & \vdots \\
		a_{n+\tilde{q}-1}& a_{n+q} &\ldots & a_{n+2\tilde{q}-2}
	\end{vmatrix},
\end{equation*}
introduced in \cite{pommerenke1966} and  has been studied by several authors. It also plays an important role in the study of singularities (see \cite{diens1957}). Noor \cite{noor1983} studied the rate of growth of $H_{\tilde{q}}(n)$ as $n \rightarrow \infty$ for functions in $\mathcal{S}$  with bounded boundary. Different choices of $\tilde{q}$ and $n$ yields various types of Hankel determinants, such as for $\tilde{q}=2$ and $n=1$, the famous Fekete-Szeg\"o functional is given by $H_2(1):=a_3-a_2^2$. Furthermore, the generalized Fekete-Szeg\"o functional is given by $a_3-\mu a_2^2$, where $\mu$ is either real or complex. For $\tilde{q}=n=2$, we have second order Hankel determinant $H_2(2):=a_2 a_4-a_3^2$. Also, another type of second order Hankel determinant is obtained by taking $\tilde{q}=2$ and $n=3$, mathematically written as  $H_2(3) := a_3a_5-a_4^2$. The estimations of the sharp bounds for these $H_{\tilde{q}}(n)$ is obtained by many authors for various sub-classes of $\mathcal{A}$ (see~\cite{cho,raza2013,zaprawa2016}). Third order Hankel determinant, given by 
\begin{equation}\label{h3}
	H_3(1) = a_3(a_2 a_4-a_3^2)-a_4(a_4-a_2 a_3)+a_5(a_3-a_2^2),\end{equation} is obtained when $\tilde{q}=3$ and $n=1$. Sharp bound of $|H_3(1)|$ was not obtained for any class of analytic functions before 2018. It is achieved by Kowalczyk $et$ $al.$~\cite{kowalczyk22018}, for functions in $\mathcal{A}$ satisfying $\RE (f(z)/z) > \alpha \text{, }\alpha \in [0,1)$ and in~\cite{kowalczyk2018} for convex functions. Following which, Banga and Kumar~\cite{banga} recently derived sharp bound of third Hankel determinant as $|H_3(1)| \leq 1/36$ for functions in $\mathcal{SL^*}$, which earlier was calculated to be $43/576$ in~\cite{raza2013}. Lecko $et$ $al.$~\cite{lecko} found that $1/9$ is the
sharp bound of the third Hankel determinant for starlike functions of order $1/2$. The credit of initiation of sharp bound of $|H_3(1)|$ goes
to Kwon $et$ $al.$~\cite{kwon2018} who deduced $p_4$ in terms of $p_1$, where $p_i's$ are the coefficients of the functions in the Carath\'{e}odory class $\mathcal{P}$, defined by:
\begin{align*}
	p(z)= 1+ p_1 z+p_2 z^2+p_3 z^3+p_4 z^4+\cdots \text{ } (z \in \mathbb{D}).
\end{align*}

Let us recall the $q-$derivative of a complex valued function defined on a subset of $\mathbb{C}$, defined as below 
$$ (D_qf)(z) = \begin{cases}
	\dfrac{f(z)-f(qz)}{(1-q)z}, & z \neq 0 \\
	f'(0), & z=0,
\end{cases}$$ where $q \in (0,1).$ Whenever $f$ is differentiable on a given subset of $\mathbb{C}$, the above definition of $q-$derivative implies
$$ \lim_{q\rightarrow 1^{-}}(D_qf)(z) = \lim_{q\rightarrow 1^{-}} \dfrac{f(z)-f(qz)}{(1-q)z} = f'(z).$$ Furthermore, Taylor series expansion of $f$ yields that
$$ (D_qf)(z) = 1+\sum_{n=2}^{\infty}[n]_q a_n z^{n-1},$$ where 
$$ [n]_q = 
\sum_{k=0}^{n-1} q^k = 1+q+q^2+\cdots + q^{n-1},~n \in \mathbb{N}.$$ The initiation of the above defined $q-$calculus is done by Jackson~\cite{jack}. In Geometric Function Theory, sub-classes of normalized analytic functions have been studied with different view points. Ismail $et$ $al.$~\cite{ismail} generalized the class $\mathcal{S}^*$ of starlike functions by introducing a new class with the usage of $q-$calculus. This marked the beginning of introduction of $q-$version of various classes in Geometric Function Theory. For instance, Srivastava and Bansal~\cite{sriq} study a certain family of $q-$Mittag-Leffler functions and Mahmood $et$ $al.$~\cite{mahq} dealt  $q-$starlike functions associated with conic domains. Recently, Khan $et$ $al.$~\cite{khanq} used $q-$derivative operator to define a new subclass of starlike functions related with the lemniscate of Bernoulli, given as
\begin{align*}
	\mathcal{SL}^*_q :=\left\{f \in \mathcal{A}: \dfrac{z (D_qf)(z)}{f(z)} \prec \sqrt{\dfrac{2(1+z)}{2+(1-q)z}},~~z \in \mathbb{D}\right\}, \end{align*} or equivalently, a function $f \in \mathcal{A}$ is in $\mathcal{SL}^*_q$ if it satisfies the following 
$$  \left|\left(\dfrac{z(D_qf)(z)}{f(z)}\right)^2-\dfrac{1}{1-q}\right|< \dfrac{1}{1-q}.$$ This implies on choosing $\omega =  z(Dq f )(z)/f (z)$, the analytic characterization of the class $\mathcal{SL}^*_q$ can be expressed as $|\omega^2-1/(1-q)|<1/(1-q),$ which is the interior of the right loop of the lemniscate of Bernoulli.	
The specialty of this class lies in the fact that it reduces to a well-known class $\mathcal{SL^*}$, when $q\rightarrow 1^-$. The authors in~\cite{khanq} obtained the sharp bounds of Fekete-Szeg\"o functional, $|H_2(2)|$, initial coefficients $a_2$, $a_3$, $a_4$ and $a_5$ and upper bound of third Hankel determinant for functions in $\mathcal{SL}^*_q$.

Our study focuses on the estimation of sharp bound of $|H_3(1)|$ for functions in $\mathcal{SL}^*_q$ and $\mathcal{S}^*$. It was found in~\cite{baba} that $|H_3(1)| \leq 16$ for functions in $\mathcal{S}^*$, which is improved by Zaprawa~\cite{zaprawa2017}, wherein he proved $|H_3(1)| \leq 1$. Later in~\cite{kwon2019}, it was further improved to $8/9$. Again, in 2021, Zaprawa $et$ $al.$~\cite{zaprawa21} calculated the same to be $5/9$, to which we eventually improve in the present paper to a sharp estimate of $4/9$. In addition, we obtain $|H_3(1)| \leq \tfrac{(1+q)^2}{16 q^2(1+q+q^2)^2}$ for functions in $\mathcal{SL^*}_q$. This bound apart from being sharp is an improvement over the bound obtained in~\cite{khanq}. Moreover, for $q \rightarrow 1^-$, this bound reduces to earlier known sharp bound for $\mathcal{SL}^*$~\cite{banga}. We also give extremal functions to justify our claims.


We state below a lemma for the formulae of $p_2$, $p_3$ \cite{libera1982} and $p_4$ \cite{kwon2018} in order to prove our results.
\begin{lemma} \label{p1p2p3}
	Let $p \in \mathcal{P}$ and of the form $1+\sum\limits_{n=1}^{\infty}p_nz^n.$ Then
	\begin{align*}
		2p_2=p_1^2+\lambda(4-p_1^2),
	\end{align*}
	\begin{align*}
		4p_3=p_1^3+2p_1(4-p_1^2)\lambda-p_1(4-p_1^2)\lambda^2+2(4-p_1^2)(1-|\lambda|^2)\mu
	\end{align*}
	and
	\begin{align*}
		8p_4=&p_1^4+(4-p_1^2)\lambda(p_1^2(\lambda^2-3\lambda+3)+4\lambda)\nonumber\\&-4(4-p_1^2)(1-|\lambda|^2)(p_1(\lambda-1)\mu+\overline{\lambda}\mu^2-(1-|\mu|^2)\delta),
	\end{align*}
	for some $\delta$, $\lambda$ and $\mu$ such that $|\delta|\leq1$, $|\lambda|\leq1$ and $|\mu|\leq 1.$\end{lemma}
\section{Main Results}

This section begins with the following result.
\begin{theorem}\label{qstar}
	Let $q \in (0,1)$ and $f \in \mathcal{SL}^*_{q}$ of the form $f(z) = z+\sum_{n=2}^{\infty}a_n z^n.$ Then we have
	\begin{align*}
		|H_3(1)| \leq \dfrac{(1+q)^2}{16 q^2 (1+q+q^2)^2}.\end{align*}\end{theorem}
\begin{proof} For $f \in \mathcal{SL}^*_{q}$, we refer~\cite{khanq} for the expressions of $a_2$, $a_3$ and $a_4.$ On the similar lines, we compute
	\begin{align*} a_5 = &~\dfrac{1}{32768 q^4 (1 + q^2)(1 + q + q^2)}\bigg(512 p_1 p_3 q^2 (2 - 10 q - 8 q^2 - 9 q^3 + 3 q^4)  + 
		p_1^4 (8- 140 q + 802 q^2\\&\quad - 1435 q^3  - 340 q^4 - 1193 q^5 + 
		1015 q^6 - 320 q^7 + 35 q^8) + 
		32 p_1^2 p_2  q (6 - 68 q + 175 q^2 + 89 q^3\\&\quad  + 148 q^4 - 93 q^5 + 
		15 q^6) + 
		256 q^2 (1 + q + q^2)(16 p_4 q  + p2^2 (2 - 13 q + 3 q^2))\bigg).\end{align*}
	Now substituting the values of above $a_i's$ in \eqref{h3} with $p:=p_1 \in [0,2]$, we obtain
	\begin{align*}
	H_3(1) :=~&\dfrac{1}{4194304~q^2 (1 + q^2) (1 + q + 
		q^2)^2} \bigg(-8192 p p_2 p_3 (-14 - 28 q - 13 q^2 - 28 q^3 - 12 q^4 + 
		q^5)\\& - 512 p_1^3 p_3 (-14 - 10 q - 217 q^2  + 23 q^3 - 8 q^4 + 
		5 q^5 + q^6)   +16 p_1^4 p_2 (-31 + 1111 q - 10148 q^2 \\&+ 3026 q^3- 594 q^4 - 84 q^5  - 19 q^6 + 3 q^7)  + 
		p_1^6 (239 - 4972 q + 35429 q^2 - 13002 q^3  + 3964 q^4\\& + 370 q^5 + 
		63 q^6  - 44 q^7 + q^8)   + 
		4096 (-16 p_3^2 (1 + q)^2 (1 + q^2) + 16 p_2 p_4 (1 + q + q^2)^2 \\&+ 
		p_2^3 (-13 + q) (1 + q + q^2)^2)  + 
		256 p_1^2 (16 p_4 (-15 + q) (1 + q + q^2)^2+ 
		p_2^2 (-27 - 102 q\\& + 670 q^2 - 188 q^3 + 6 q^4 + 14 q^5 + 
		3 q^6))\bigg).
	\end{align*}
	Applying Lemma \ref{p1p2p3} in the above equation for the values of $p_2$, $p_3$ and $p_4$ and further reducing it to the simpler form, we arrive at 
	\begin{align} \label{h31} H_3(1) = \dfrac{\tau_1(p,\lambda) +\tau_2(p,\lambda) \mu + \tau_3(p,\lambda)\mu^2+\zeta(p,\lambda,\mu) \delta}{4194304~q^2 (1 + q^2) (1 + q + q^2)^2},\end{align} 
	whenever $\delta$, $\mu$, $\lambda \in \overline{\mathbb{D}}$ and 
	\begin{align*}
		&\tau_1(p,\lambda) :=  A p^6 +p^2(4-p^2)\lambda\bigg(8 (-15 + 183 q - 804 q^2 + 434 q^3 - 242 q^4 - 20 q^5 - 3 q^6 + 3 q^7)p^2   
		\\ &  \qquad \qquad \quad +64(45+66 q+262 q^2+28 q^3+62 q^4+6 q^5+3 q^6) (4-p^2)\lambda -512(7+15 q-3q^2\\ &  \qquad \qquad \quad  +17 q^3+5 q^4-q^5)(4-p^2)\lambda^2-2048(7-q)(1+q+q^2)^2\lambda +4096q^2(4-p^2)\lambda^3\\ &  \qquad \qquad \quad -512(7-q)(1+q+q^2)^2 p^2\lambda^2 +128(22+50 q+35 q^2+59 q^3+20 q^4+q^5+q^6)p^2\lambda\bigg) \\ &  \qquad \qquad \quad-2048(5-q)(1+q+q^2)^2(4-p^2)^2\lambda^3, \\ 
		&\tau_2(p,\lambda) :=  (4-p^2)(1-|\lambda|^2)\bigg(256(6+2q+41q^2-15 q^3-5 q^5-q^6)p^3+2048(7-q)(1+q+q^2)^2p^3\lambda \\ & \qquad \qquad \quad+p(4-p^2)\lambda(2048(6+12 q+5 q^2+12 q^3 +4 q^4-q^5)-16384 q^2  \lambda)\bigg), \\ 
		&\tau_3(p,\lambda) := (4-p^2)(1-|\lambda|^2)(2048(7-q)(1+q+q^2)^2p^2 \overline{\lambda}-(4-p^2)(16384 q^2|\lambda|^2\\ & \qquad \qquad \quad+16384(1+q^2)(1+q)^2)), \\ 
		&\zeta(p, \lambda,\mu) := (4-p^2)(1-|\lambda|^2)(1-|\mu|^2)(1+q+q^2)^2(-(14336-2048 q)p^2+16384(4-p^2)\lambda),
	\end{align*}
	where $A := 55 - 308 q + 1349 q^2 - 698 q^3 + 620 q^4 - 46 q^5 - 25 q^6 - 
	20 q^7 + q^8$. Taking modulus over equation~\eqref{h31} and applying triangle inequality, we get
	$$ |H_3(1)| \leq \dfrac{|\tau_1(p,\lambda)|+|\tau_2(p,\lambda)|y+|\tau_3(p,\lambda)|y^2+|\zeta(p,\lambda,\mu)|}{4194304~q^2 (1 + q^2) (1 + q + 
		q^2)^2} \leq \tilde{T}(p,x,y),$$
	where $x := |\lambda|$, $y := |\mu|$ and the fact $|\delta| \leq 1$ and
	\begin{align*}
		\tilde{T}(p,x,y) & := \dfrac{t_1(p,x)+t_2(p,x)y+t_3(p,x)y^2+t_4(p,x)(1-y^2)}{4194304~q^2 (1 + q^2) (1 + q + 
			q^2)^2} \\ & =: \dfrac{T(p,x,y)}{4194304~q^2 (1 + q^2) (1 + q + 
			q^2)^2}.
	\end{align*}
	with
	\begin{align*}
		&t_1(p,x) :=  A p^6 +p^2(4-p^2)x\bigg(8 (15 - 183 q + 804 q^2 - 434 q^3 + 242 q^4 + 20 q^5 + 3 q^6   - 3 q^7)p^2 
	\\ &  \qquad \qquad \quad	+64(45+66 q+262 q^2+28 q^3+62 q^4+6 q^5+3 q^6)(4-p^2)x +512(7+15 q-3q^2\\ &  \qquad \qquad \quad +17 q^3+5 q^4-q^5)(4-p^2)x^2+2048(7-q)(1+q+q^2)^2 x+4096(4-p^2)x^3\\ &  \qquad \qquad \quad +512(7-q)(1+q+q^2)^2 p^2x^2 +128(22 +50 q+35 q^2+59 q^3+20 q^4+q^5+q^6)p^2x\bigg) \\ &  \qquad \qquad \quad+2048(5-q)(1 +q+q^2)^2(4-p^2)^2x^3, \\ &  
		t_2(p,x) :=  (4-p^2)(1-x^2)\bigg(256(6+2q+41q^2-15 q^3-5 q^5-q^6)p^3+2048(7  -q)(1+q+q^2)^2p^3x \\ &  \qquad \qquad \quad+p(4-p^2)x(2048(6+12 q+5 q^2+12 q^3 +4 q^4-q^5)+16384 q^2 x)\bigg), \\ &
		t_3(p,x) := (4-p^2)(1-x^2)(2048(7-q)(1+q+q^2)^2p^2 x+(4-p^2)(16384 q^2 x^2 \\ &  \qquad \qquad \quad +16384(1+q^2)(1+q)^2)), \\ &
		t_4(p,x) := (4-p^2)(1-x^2)(1+q+q^2)^2((14336-2048 q)p^2 +16384(4-p^2)x).
	\end{align*}
	In order to achieve the desired bound, we need to maximize $T(p,x,y)$ in the closed cuboid $\mathfrak{C}: [0,2] \times [0,1] \times [0,1]$. We accomplish this by estimating maximum values in the interior of $\mathfrak{C}$, interior of the six faces and finally on the twelve edges.
	
	I. We begin with interior points of $\mathfrak{C}$, which means taking $(p,x,y) \in (0,2) \times (0,1) \times (0,1)$. \\
	For this, we calculate
	\begin{align*}\dfrac{\partial T}{\partial y} = &~(4-p^2)(1-x^2)(2y(16384(4-p^2)((1+q^2)(1+q)^2+q^2 x^2)+2048 p^2 x(7-q)(1+q+q^2)^2\\&-((14336-2048 q)p^2+16384(4-p^2)x)(1+q+q^2)^2)+256 (6+2q+41 q^2-15 q^3-5q^5\\&-q^6)p^3+2048(7-q)(1+q+q^2)^2p^3x+2048 p x(4-p^2)(6+12q+5q^2+12q^3+4q^4-q^5)\\&+16384 p q^2(4-p^2)x^2).
	\end{align*}
	On solving $\partial T/\partial y = 0$, we obtain $y=y_0$, given as
	$$y_0 := \dfrac{\tilde{A}}{2048(1-x)\left((7-q)(1+q+q^2)^2p^2-8(4-p^2)\left(\tfrac{1+2q+2q^2+2q^3+q^4}{q^2}-x\right)\right)},$$ where $\tilde{A}:= p^3(128(6+2q+41q^2-15q^3-5q^5-q^6)+1024(7-q)(1+q+q^2)^2x)+1024 p x(4-p^2)(6+12 q+5q^2+12q^3+4q^4-q^5+8q^2x).$ For $y_0 \in (0,1)$, we must have
	\begin{align*}(7-q)(1+q+q^2)^2p^2 > 8(4-p^2)\left(\dfrac{1+2q+2q^2+2q^3+q^4}{q^2}-x\right)\end{align*} and 
	\begin{align}\label{y0}
		\tilde{A} +16384 (4-p^2) (1-x) \left(\dfrac{1+2q+2q^2+2q^3+q^4}{q^2}-x\right) < 2048 (1-x)(7-q)(1+q+q^2)^2p^2.	
	\end{align}
	Let us assume $p \rightarrow 2$, then there exists $x \in (0,\tfrac{101}{216})$ for every $q \in (0,1)$ such that \eqref{y0} holds. Moreover when we consider $x \in [\tfrac{101}{216},1)$, then there does not exist any $p \in (0,2)$ for all $q \in (0,1)$ such that \eqref{y0} holds. Assuming $x \rightarrow 0$, we compute \eqref{y0} holds for $p \geq 1.48855$ for every $q \in (0,1).$ In fact whenever $p \in (0,1.48855)$, there does not exist any $ x \in (0,1)$ for all $q \in (0,1)$ such that \eqref{y0} holds. Thus we conclude possible solution exists in $[1.48855,1) \times (0,\tfrac{101}{216})$ for inequality \eqref{y0}. A computation shows $$\dfrac{\partial T}{\partial p}\bigg|_{y=y_0} \neq 0, $$ in this interval. Therefore, there does not exist any critical point in the interior of $\mathfrak{C}$.
	
	II. Now we compute the maximum value of $T$ in the interior of all the six faces of $\mathfrak{C}$.\\
	On the face $p=0$, $T(p,x,y)$ reduces to
	\begin{align} \label{0xy} T(0,x,y) = &~262144 (1 - x^2) ((1+ q^2) (1 + q)^2  + q^2 x^2-(1+q+q^2)^2x) y^2 \nonumber \\ &+ 
		32768 x  (1 + q + q^2)^2(x^2(5 - q)+8(1-x^2)),\end{align} which in turn differentiating with respect to $y$ becomes
	$$\dfrac{\partial T}{\partial y}=524288 y (1-x^2)(x-1)\bigg(x-\dfrac{1+2q+2q^2+2q^3+q^4}{q^2}\bigg) \neq 0 \quad x,y \in (0,1).$$ This clearly shows there does not exist any critical point for $T(0,x,y)$ in $(0,1)\times (0,1)$.
	
	On the face $p=2$, $$\tilde{T}(p,x,y)=\tilde{T}(2,x,y)= \dfrac{A}{65536  q^2 (1+q^2)(1+q+q^2)^2} \leq  \dfrac{(1+q)^2}{16 q^2 (1+q+q^2)^2},$$ $x,y \in (0,1)$, as we have $-4041 -8500 q - 6843 q^2 - 8890 q^3 - 3476 q^4 - 46 q^5 - 25 q^6 - 20 q^7 + q^8 \leq 0$ for $q \in (0,1).$
	
	On the face $x=0$, $T(p,x,y)$ becomes 
	\begin{align}\label{p0y2}
	T(p,0,y)=&~p^6 (55 - 308 q + 1349 q^2 - 698 q^3 + 620 q^4 - 46 q^5 - 25 q^6 - 20 q^7 + q^8) \nonumber \\ &+ 
		256 (4 - p^2) (-p^3 (-6 - 2 q - 41 q^2 + 15 q^3 + 5 q^5 + q^6) y  + 
		64 (4 - p^2)  (1 \nonumber \\ &+ q)^2 (1 + q^2) y^2 + 
		8 p^2 (-7 + q) (1 + q + q^2)^2 (-1 + y^2)):=h_1(p,y).
	\end{align}
	On solving $\tfrac{\partial h_1}{\partial y}=0$, we get
	\begin{equation} \label{y1} y=: y_1=  \dfrac{p^3 (-6 - 2 q - 41 q^2 + 15 q^3 + 5 q^5 + 
			q^6)}{16 (32 (1 + q)^2 (1 + q^2) + 
			p^2 (-15 - 29 q - 35 q^2 - 27 q^3 - 13 q^4 + q^5))}.\end{equation} For $0<p\leq 1.46$, we have $y_1\leq 0$ for every $q \in (0,1)$. There exists some $q \in (0,1)$ whenever  $p \in (1.46, 2)$ such that $y_1>0$.  On substituting \eqref{y1} in $\tfrac{\partial h_1}{\partial p}$ and simplifying further, we get $\tfrac{\partial h_1}{\partial p} \neq 0$, where 
	$p \in (1.46,2)\text{, } q\in (0,1)$.
	Thus $h_1(p,y)$ has no critical point in $(0,2) \times (0,1)$.
	
	On the face $x=1$, $T(p,x,y)$ reduces to  
	\begin{align} \label{p1y}  T(p,1,y)=: & -32768 (-5 + q) (1 + q + q^2)^2 + 
		p^6 (1 + q)^3 (-1 + 7 q + 19 q^2   + 9 q^3\nonumber \\& + q^4  + q^5)  + 
		1024 p^2 (77 + 146 q + 246 q^2 + 140 q^3 + 94 q^4 + 6 q^5 + 3 q^6) \nonumber \\& - 
		32 p^4(929 + 1783 q  + 2636 q^2 + 1666 q^3 + 878 q^4 - 4 q^5 + 
		29 q^6 + 3 q^7) =: h_2(p).\end{align} 
	On differentiating $h_2$ with respect to $p$, we obtain
	\begin{align*} \dfrac{\partial h_2}{\partial p} = &~6 p^5 (1 + q)^3 (-1 + 7 q + 19 q^2 + 9 q^3 + q^4 + q^5) + 
		2048 p (77 + 146 q + 246 q^2 + 140 q^3 + 94 q^4\\&+ 6 q^5 + 3 q^6) - 
		128 p^3 (929 + 1783 q + 2636 q^2 + 1666 q^3 + 878 q^4- 4 q^5 + 
		29 q^6 + 3 q^7),\end{align*}
	further which becomes $0$ at $p=0$ and $p=p_0$, given by
	$$p_0 := \sqrt{\dfrac{32 (929 + 1783 q + 2636 q^2 + 1666 q^3 + 878 q^4 - 4 q^5 + 
			29 q^6 + 3 q^7)}{
			3 (1 + q)^3 (-1 + 7 q + 19 q^2 + 9 q^3 + q^4 + q^5)} - \tilde{A}},$$ where $\tilde{A} = \dfrac{64 \sqrt{2}A_0}{3 (-1 + 4 q + 37 q^2 + 86 q^3 + 92 q^4 + 50 q^5 + 15 q^6 + 4 q^7 + 
		q^8)}$ and \begin{align*}A_0= (&
	107909 + 414041 q + 1008402 q^2 + 1557100 q^3 + 1804144 q^4 + 
	1471838 q^5+ 913014 q^6 \\&+ 363176 q^7 + 107408 q^8 + 9900 q^9 + 
	6570 q^{10} + 346 q^{11} + 41 q^{12} + 15 q^{13})^{(1/2)}.\end{align*} A calculation yields $p=0$ is a point of minima and $p_0$ is a point of maxima and maximum value is given by a huge mathematical expression in $q$ which is computed to be less than $262144(1+q)^2(1+q^2)$.			
	
	On the face $y=0$, we have $T(p,x,0) =: h_3(p,x)$, given by
	\begin{align*} h_3(p,x):= & p^6 (55 - 308 q + 1349 q^2 - 698 q^3 + 620 q^4 - 46 q^5 - 25 q^6- 20 q^7 + q^8)\\ & + 2048 (4 - p^2) (1 + q + q^2)^2 (-1 + x^2) (-32 x + p^2 (-7 + q + 8 x)) \\& +8 (4 - p^2)x (-1024 (-5 + q) (1 + q + q^2)^2 x^2 + 
	32 p^2 x (101 - 2 q^5\\& + 3 q^6 + 16 x + 2 q^4 (51 + 8 x) + 
	4 q^3 (29 + 20 x) + 2 q (85 + 24 x)\\& + 
	2 q^2 (207 - 64 x + 32 x^2)) - 
	p^4 (-15 + 3 q^7 + 8 x + 22 q^4 (-11\\& + 8 x) + 4 q^5 (-5 + 8 x) + 
	q^6 (-3 + 8 x) + q (183 - 272 x + 128 x^2) \\&+ 
	q^3 (434- 720 x + 384 x^2) + 
	4 q^2 (-201 + 384 x - 352 x^2 + 128 x^3))) . \end{align*} A calculation yields that there is no common solution to the system of equations $\tfrac{\partial h_3}{\partial x}=0$ and $\tfrac{\partial h_3}{\partial p}=0$ in $(0,2) \times (0,1).$ Similarly we can show that there does not exist any critical point for $T(p,x,1)$.
	
	III. Finally, we estimate the maximum value on the edges of the cuboid $\mathfrak{C}$.
	
	\noindent Starting with the $T(p,0,0) =: h_4(p)$, given by
	\begin{align*} h_4(p) = & (1 + q + q^2)^2 (4 p^2 (14336 - 2048 q) - 
		p^4 (14336 - 2048 q)) + (55 \\ & - 308 q  + 1349 q^2 - 698 q^3 + 
		620 q^4 - 46 q^5 - 25 q^6 - 20 q^7 + q^8) p^6,\end{align*} obtained from \eqref{p0y2}. On solving $\tfrac{\partial h_4}{\partial p} =0$, we get 
	either $p=0$ or $p =: p_0$, given by
	\begin{align*} p_0 := &\dfrac{1}{\sqrt{3 A}} \bigg(2048 (7 - q) (1 + q + q^2)^2-64 \sqrt{2} \bigg(
		23933 + 97507 q  +  203268 q^2 + 309564 q^3 + 313752 q^4 \\& + 
		250248 q^5 + 114774 q^6  + 34938 q^7 - 6156 q^8  - 644 q^9 + 
		1352 q^{10} + 192 q^{11} - 75 q^{12} + 
		3 q^{13}\bigg)^{1/2}\bigg)^{1/2}.\end{align*} We compute that the function $h_4(0)=0$ is a minimum value of $h_4(p)$ and $h_4(p_0)$ is a huge mathematical expression in $q$ which is also a maximum value of $h_4(p)$. Further we have $\tilde{T}(p_0,0,0) \leq (1+q)^2/16 q^2 (1+q+q^2)^2$. Substituting $y=1$ in equation~\eqref{p0y2}, we obtain
	\begin{align*} 
		T(p,0,1) = h_5(p) = &(4 - p^2) (-256 (-6 - 2 q - 41 q^2 + 15 q^3 + 5 q^5 + q^6) p^3 \\&  +  (4 - 
		p^2) (16384 (1 + q^2) (1 + q)^2)) + A p^6 .\end{align*} The function $h_5(p)$ is a decreasing function of $p$ for all $q$. Thus
	$$ \max_{p \in [0,2]} \tilde{T}(p,0,1) = \tilde{T}(0,0,1) = \dfrac{(1+q)^2}{16 q^2 (1+q+q^2)^2}.$$ 
	Form equation~\eqref{p1y}, which is independent of $y$, we get
	$\tilde{T}(p,1,0) = \tilde{T}(p,1,1) = \tilde{T}(p,1,y)$. Thus $\tilde{T}(p,1,0) =\tilde{T}(p,1,1) \leq \tfrac{(1+q)^2}{16 q^2 (1+q+q^2)^2}.$ Substituting $x=0$ in ~\eqref{0xy}, we obtain
	$\tilde{T}(0,0,y) = y^2(1+q)^2/16 q^2(1+q+q^2)^2,$ which is clearly an increasing function of $y$ for all $q$ and we have
	$$\tilde{T}(0,0,y) \leq \tilde{T}(0,0,1) = \dfrac{(1+q)^2}{16 q^2 (1+q+q^2)^2}.$$ Evaluating~\eqref{p1y} at $p=0$, we get
	$$ \tilde{T}(0,1,y) = \dfrac{5-q}{128 q^2 (1+q^2)}.$$ The value of $\tilde{T}(p,x,y)$ on the edges $p=2$, $x=1$; $p=2$, $x=0$; $p=2$, $y=0$; $p=2$, $y=1$ are all respectively equal to $\tilde{T}(2,1,y)=\tilde{T}(2,0,y) =\tilde{T}(2,x,0)=\tilde{T}(2,x,1)= \tilde{T}(2,x,y)$  as $\tilde{T}(2,x,y)$ is independent of both $x$ and $y$, which further equals to
	\begin{align*}  \dfrac{A}{65536  q^2 (1+q^2)(1+q+q^2)^2} \leq \dfrac{(1+q)^2}{16 q^2 (1+q+q^2)^2}. \end{align*} Evaluating equation~\eqref{0xy} at $y=0$, we deduce
	$$T(0,x,0) = h_6(x) = 32768 (1 + q + q^2)^2 x (8 - (3 + q) x^2).$$
	On solving $h_6'(x) =0$, we get 
	$$ x = x_0 := \dfrac{512  (1 + q + q^2)}{\sqrt{294912 + 688128 q + 1081344 q^2 + 
			884736 q^3 + 491520 q^4 + 98304 q^5}}.$$ A computation shows that $x _0$ is a point of maxima and maximum value is given by
	$$ \max_{x \in [0,1]} h_6(x) = h_6(x_0) = \dfrac{\sqrt{2}(1+q+q^2)}{12\sqrt{
			3 (3 + q)}}\text{ }(0 < q < 1).$$ Also, we have
	$$ \max_{0\leq x\leq1} \tilde{T}(0,x,0) \leq \dfrac{(1+q)^2}{16 q^2 (1+q+q^2)^2}.$$ Now evaluating~\eqref{0xy} at $y=1$, we obtain
	$$ T(0,x,1) = 262144 (1 - x^2) ((1 + q^2) (1 + q)^2 + q^2 x^2) + 
	32768 x^3 (5 - q) (1 + q + q^2)^2,$$ which is clearly a decreasing function of $x$ and attains maximum value at $x=0$, given by $\tfrac{(1+q)^2}{16 q^2(1+q+q^2)^2}$. \\
	Altogether I-III, yield $|H_3(1)| \leq \tfrac{(1+q)^2}{16 q^2(1+q+q^2)^2}.$ The result is sharp as equality occurs for the function $\tilde{f}: \mathbb{D} \rightarrow \mathbb{C}$, satisfying the following
	\begin{equation*}
		\dfrac{z (D_q\tilde{f})(z)}{\tilde{f}(z)} = \sqrt{\dfrac{2(1+z^3)}{2+(1-q)z^3}}.
\end{equation*}\end{proof}
Let $q\rightarrow 1^-$ in the above Theorem,  then it reduces to the following result obtained by Banga and Kumar~\cite{banga}.
\begin{corollary}
	Let $f \in \mathcal{SL}^*$. Then $|H_3(1)|  \leq 1/36$.
\end{corollary}
Moreover,  extremal functions also coincide in the case of $q \rightarrow 1^-$.
\begin{theorem}
	Let $f \in \mathcal{S}^*$ of the form $f(z)= z +\sum_{n=2}^{\infty} a_n z^n$. Then the sharp bound for third order Hankel determinant for such functions is given by 
	\begin{align}\label{bound}
		|H_3(1)| \leq 4/9.
\end{align}\end{theorem}
\begin{proof} For $f \in \mathcal{S}^*$, we have
	\begin{align}\label{stareq}
		\dfrac{z f'(z)}{f(z)} = \dfrac{1+\omega(z)}{1-\omega(z)},
	\end{align}
	for some Schwarz function $\omega(z)$. Define a function, $p(z) =\tfrac{1+\omega(z)}{1-\omega(z)}$, then evidently $p \in \mathcal{P}$. The equation~\eqref{stareq} now reduces to
	\begin{align*}
		\dfrac{z f'(z)}{f(z)} = p(z) = 1+p_1 z+p_2 z^2+p_3 z^3+\cdots.
	\end{align*} 
	The Taylor series of which yield
	\begin{align*}
		a_2=p_1\text{, }a_3=\dfrac{p_2+p_1^2}{2}\text{, }a_4=\dfrac{p_1^3+3 p_1 p_2+ 2p_3}{6}\end{align*} and \begin{align*}				a_5=\dfrac{p_1^4+6 p_1^2p_2+3 p_2^2+8 p_1p_3+ 6 p_4}{24}.\end{align*}
	Here, we assume $p_1=:p$ lies in the interval $[0,2]$ due to the invariant property of class $\mathcal{P}$ under rotation. Equation~\eqref{h3}, together with the above expressions of $a_i's$, yield
	\begin{align*}
		H_3(1)=\dfrac{- p^6 +3 p^4 p_2 +8 p^3 p_3 + 24 p p_2 p_3 -9p^2 p_2^2 - 18 p^2 p_4 -9 p_2^3 - 16 p_3^2 +18 p_2 p_4}{144}.
	\end{align*}
	Applying Lemma \ref{p1p2p3} in the above equation for the values of $p_2$, $p_3$ and $p_4$ and further reducing it to the simpler form, we arrive at
	\begin{align*}
		H_3(1)=\dfrac{1}{1152}\bigg(\tau_1(p,\lambda)+\tau_2(p,\lambda)\mu+\tau_3(p,\lambda)\mu^2+\varsigma(p,\lambda,\mu)\delta\bigg).
	\end{align*}
	Where $\delta$, $\mu$, $\lambda$ $\in \overline{\mathbb{D}},$
	\begin{align*}
		&\tau_1(p,\lambda):=-2p^2\lambda^2(4-p^2)^2-10p^2\lambda^3(4-p^2)^2+p^2\lambda^4(4-p^2)^2\\& \qquad \qquad \quad+3p^4\lambda(4-p^2)+3p^4\lambda^2(4-p^2) -36p^2\lambda^2(4-p^2)-9p^4\lambda^3(4-p^2),\\&\tau_2(p,\lambda):=(4-p^2)(1-|\lambda|^2)\left(12 p^3+36 p^3\lambda+p \lambda (4-p^2)(20 -4 \lambda)\right),\\& \tau_3(p,\lambda):=(4-p^2)(1-|\lambda|^2)\left(36p^2\overline{\lambda}-4(4-p^2)(|\lambda|^2+8)\right),\\
		&\varsigma(p,\lambda,\mu):= (4-p^2)(1-|\lambda|^2)(1-|\mu|^2)\left(-36p^2+36\lambda(4-p^2)\right).
	\end{align*}
	Assuming $x:=|\lambda|$, $y:=|\mu|$ and using the fact $|\delta|\leq 1$, we have
	\begin{align*}
		|H_3(1)| \leq \dfrac{|\tau_1(p,\lambda)|+|\tau_2(p,\lambda)|y+|\tau_3(p,\lambda)|y^2+|\varsigma(p,\lambda,\mu)|}{1152}\leq S(p,x,y),
	\end{align*}
	where
	\begin{align}\label{G}
		S(p,x,y):=\dfrac{1}{1152}\bigg(s_1(p,x)+s_2(p,x)y+s_3(p,x)y^2+s_4(p,x)(1-y^2)\bigg)\end{align} with
	\begin{align*}
		&s_1(p,x):=2p^2x^2(4-p^2)^2+10 p^2 x^3(4-p^2)^2+p^2x^4(4-p^2)^2 +3p^4x(4-p^2)\\& \qquad \qquad \quad+3p^4x^2(4-p^2)+36p^2x^2(4-p^2)+9p^4x^3(4-p^2),\\&s_2(p,x):=(4-p^2)(1-x^2)(12p^3+px(4-p^2)(20+4x)+36p^3x),\\&s_3(p,x):=(4-p^2)(1-x^2)(32(4-p^2)+4x^2(4-p^2)+36p^2x),\\ &s_4(p,x):=(4-p^2)(1-x^2)(36p^2+36x(4-p^2)).\end{align*}
	Our aim is to maximize $S(p,x,y)$ in the closed cuboid $\mathfrak{C}: [0,2] \times [0,1] \times [0,1]$. We accomplish this by obtaining the maximum values in the interior of  $\mathfrak{C}$, in the interior of the six faces and on the twelve edges. 
	
	I.  First we consider the interior points of $\mathfrak{C}.$ Let $(p,x,y) \in (0,2)\times(0,1)\times(0,1)$. In order to achieve the maximum value in the interior of $\mathfrak{C}$, we partially differentiate equation (\ref{G}) with respect to $y$ and further reduce it to a simpler expression as		\begin{align*}
		\dfrac{\partial S}{\partial y}=\dfrac{1}{1152}(4-p^2)(1-x^2)(8y(x-1)((4-p^2)(x-8)+9 p^2)+4p(x(4-p^2)(5+x)+p^2(3+9x))).
	\end{align*}
	Now $\tfrac{\partial S}{\partial y}=0$ yields
	\begin{align*}
		y=:y_0=\dfrac{2p(x(4-p^2)(5+x)+p^2(3+9x))}{(1-x)((4-p^2)(x-8)+9p^2)}.
	\end{align*}
	In order to find the critical points, we first ensure $y_0$ should lie in the interval $(0,1)$, which is possible only when
	\begin{align}\label{cond}
		p^3(6+18x)+2px(4-p^2)(5+x)+(1-x)(8-x)(4-p^2) < 9 p^2(1-x)
	\end{align}
	and
	\begin{align}\label{cond2}
		9p^2 > (4-p^2)(8-x). 
	\end{align}
	We determine the common solutions for the above inequalities. A computation shows that inequality~\eqref{cond2} holds for all $x \in (0,1)$ whenever $ p > 1.37199$, but inequality~\eqref{cond} does not hold in $(0,2) \times (0,1)$. Therefore the function $S$ has no critical point in the given domain of values.
	
	II. Here below we calculate the maximum value on the six faces of the cuboid $\mathfrak{C}.$\\
	On the face $p=0,$ $S(p,x,y)$ becomes
	\begin{align}\label{f4}
		h_1(x,y):= S(0,x,y)=\dfrac{(1-x^2)(y^2(x-1)(x-8)+9x)}{18},
	\end{align}
	where $ x, y \in (0,1)$. We calculate
	\begin{align*}
		\dfrac{\partial h_1}{\partial y} = \dfrac{(1-x^2)y}{9}\left((x-1)(x-8)\right) \neq 0, \quad x,y \in (0,1).
	\end{align*}
	Clearly, we can infer from above that $h_1$ has no critical point in $(0,1) \times (0,1)$.
	
	On the face $p=2$, $S(p,x,y)$ becomes
	\begin{align}\label{2xy}
		S(2,x,y) =0, \quad x,y \in (0,1).
	\end{align}
	
	On the face $x=0$, $S(p,x,y)$ becomes
	\begin{align}\label{p0y}
		S(p,0,y) =: h_2(p,y) = \dfrac{(4-p^2)}{288}\left(3 p^3 y + y^2(8(4-p^2)-9 p^2)+ 9 p^2\right),	
	\end{align} $y \in (0,1) \text{ and } p \in (0,2).$
	Now we differentiate $h_2(p,y)$ partially with respect to y and obtain
	\begin{align*}
		\dfrac{\partial h_2}{\partial y} = \dfrac{(4-p^2)}{288} \left(3 p^3 +2y(8(4-p^2)-9 p^2)\right), \quad p \in (0,2)\text{ and } y \in (0,1).\end{align*}
	On solving $\partial h_2/\partial y =0$, we get 
	\begin{align}\label{yx0}
		y= \dfrac{3 p^3}{2(17 p^2 -32)},
	\end{align}
	which belongs to $(0,1)$ only when $p > p_0 \approx 1.47073$. Upon substituting the value of $y$ from equation~\eqref{yx0} in $\partial h_2/\partial p =0$, we arrive at 
	\begin{align*}
		\dfrac{(p (16384 - 25600 p^2 + 12944 p^4 - 2048 p^6 - 51 p^8)}{
			64 (32 - 17 p^2)^2}=0,
	\end{align*}
	for $p=1.20671$ in $(0,2)$. Thus there does not exist any critical point of $h_2$ in $(0,2) \times (0,1)$. 
	
	On the face $x=1$, $S(p,x,y)$ becomes
	\begin{align}\label{p1y1}
		S(p,1,y) =: h_3(p) = \dfrac{p^2}{576}\left(176-40 p^2-p^4\right). 
	\end{align}
	To find the maximum value of $h_3$, we solve $\partial h_3/\partial p =0$, which implies $p =: p_0 \approx 1.42948$ in $(0,2)$. A further calculation reveals $h_3''(p_0)<0$, indicating $p_0$ is the point of maxima and
	$$ S(p,1,y) \leq S(p_0,1,y) \approx 0.319595\text{, } p \in (0,2)\text{ and } y \in (0,1).$$
	
	On the face $y=0$, $S(p,x,y)$ becomes
	$$S(p,x,0)=: h_4(p,x)= (4-p^2) (144 x (1-x^2) + p^4 x (3 + x - x^2 - x^3) +4 p^2 (9 - 9 x + 2 x^2 + 19 x^3 + x^4)).$$ A computation yields
	\begin{align*}\dfrac{\partial h_4}{\partial p} = & 2 p (3 p^4 x (-3 - x + x^2 + x^3) + 16 (9 - 18 x + 2 x^2 + 28 x^3 + x^4)\\& - 8 p^2 (9 - 12 x + x^2 + 20 x^3 + 2 x^4))\end{align*}
	and
	$$\dfrac{\partial h_4}{\partial x} = (4 - p^2) (144 (1-3 x^2) + p^4 (3 + 2 x - 3 x^2 - 4 x^3) + 4 p^2 (-9 + 4 x + 57 x^2 + 4 x^3)).$$ We observe that there is no common solution for the equations $\tfrac{\partial h_4}{\partial p}=0$ and $\tfrac{\partial h_4}{\partial x} =0$, which indicates there does not exist any critical point of $h_4(p,x)$ in $(0,2) \times (0,1)$. 
	
	On the face $y=1$, $S(p,x,y)$ becomes $S(p,x,1)$, given as
	\begin{align*}
		h_5(p,x) := & \dfrac{1}{1152}\bigg( (1-x^2)(512 + 64 x^2  + 16 p x (20 + 4 x) + 
		p^2 (176 x^2 + 160 x^3 + 16 x^4 - 256 + 144 x \\& - 
		32 x^2)+ 
		p^4 (12 x - 40 x^2 - 44 x^3 - 8 x^4 + 32 - 
		36 x   + 4 x^2)  + 
		p^3 (48  + 144 x  - 8 x (20 + 4 x)) \\&+ 
		p^5 (-12 - 36 x  + x (20 + 4 x)) + 
		p^6 (-3 x- x^2 + x^3 + x^4)) \bigg). 
	\end{align*}
	On solving $\tfrac{\partial h_5(p,x)}{\partial x} = 0$ and $\tfrac{\partial h_5(p,x)}{\partial p} =0,$ we observe that there is no common solution to these equations, hence there does not exist any critical point of $h_5$ in $(0,2) \times (0,1)$.

	III. Finally, we find the maximum values attained by $S(p,x,y)$ on the edges of the cuboid $\mathfrak{C}.$ The equations~\eqref{f4},\eqref{2xy},\eqref{p0y} and \eqref{p1y1} are appropriately used  to evaluate $S(p,x,y)$ below for particular values of $p$, $x$ and $y$.
	\begin{itemize}
		\item[(i)] $S(p,0,0) = p^2(4-p^2)/32=: l_1(p)$. Now, $l_1'(p)=0$ for $p=0$ and $p=: \gamma_0 = \sqrt{2}$. Simply by second derivative test, we obtain $p=0$ is the point of minima and maximum value $1/8$ is attained at $\gamma_0$. So, we have
		$$ S(p,0,0) \leq \dfrac{1}{8}, \quad p \in [0,2].$$ 
		\item[(ii)] $S(p,0,1) = (4 - p^2) (32 - 8 p^2 + 3 p^3)/288$, which is a decreasing function of $p$ in the given range of $p$. Thus maximum value is obtained at $p=0$ and 
		$$ S(p,0,1) \leq S(0,0,1) = \dfrac{4}{9}\text{, } p \in [0,2].$$
		\item[(iii)] Since $S(p,1,y)$ is independent of $y$, we obtain $S(p,1,0) = S(p,1,1) = p^2 (176 - 40 p^2 - p^4) /576 =h_3(p)$, given in~\eqref{p1y1}. Thus
		$$S(p,1,0) = S(p,1,1) \leq 0.319595 \text{, } p \in [0,2].$$
		\item[(iv)] $S(0,0,y) = 4 y^2/9$, clearly which attains maximum value $4/9$ at $y=1$. So 
		$$S(0,0,y) \leq \dfrac{4}{9}\text{, } y \in [0,1].$$
		\item[(v)] $S(0,1,y) = S(2,0,y) = S(2,1,y) = 0$, $y \in [0,1]$.
		\item[(vi)] $S(0,x,0) = x(1-x^2)/2 =: l_3(x)$. Now $l_3'(x)= (1-3 x^2)/2=0$ gives $x= \gamma_1 := 1/\sqrt{3}$ in the interval $[0,1]$. Further second derivative of $l_3(x)$ is negative at $\gamma_1$. Thus $\gamma_1$ is the point of maxima and
		$$S(0,x,0) \leq \dfrac{1}{3 \sqrt{3}} = 0.19245\text{, } x \in [0,1].$$
		\item[(vii)] $S(0,x,1) = (1-x^2)(x^2+8)/18$, which is a decreasing function of $x$ in $[0,1]$. So clearly maximum value is attained at $x=0$ and we have
		$$ S(0,x,1) \leq  \dfrac{4}{9}\text{, } x \in [0,1].$$
		\item[(viii)] $S(2,x,0) = S(2,x,1) = 0$, $x \in [0,1]$.
	\end{itemize}	
	\indent Taking into account these cases I-III, the inequality~\eqref{bound} is proved. Consider the function $\tilde{f}: \mathbb{D} \rightarrow \mathbb{C}$ as follows
	\begin{align*}
		\tilde{f}(z)=z \exp\left(\int_{0}^{z}\dfrac{\left(\tfrac{1+t^3}{1-t^3}\right)-1}{t}dt
		\right)=z+\dfrac{2z^4}{3}+\cdots,\end{align*} clearly belongs to  $\mathcal{S}^*$ and for which, we have $a_2=a_3=a_5=0$ and $a_4=2/3$. This shows the bound $|H_3(1)|$ is sharp as equation~\eqref{h3} yields $|H_3(1)|= 4/9$ for this function.  \end{proof}

	%
	\subsection*{Acknowledgment}
The first author is supported by a Research Fellowship from the Department of Science and Technology, New Delhi (Ref No. IF170272).


\begin{thebibliography}{99}
		\bibitem{ali2012}
		Ali,~R.M., Cho,~N.E., Ravichandran,~V., Kumar,~S.S., \textit{Differential subordination for functions associated with the lemniscate of Bernoulli}, 
		Taiwanese J. Math. \textbf{16}(2012), no.~3, 1017--1026.
		\bibitem{aliravi2012}
		Ali,~R.M., Jain,~N.K., Ravichandran,~V., \textit{Radii of starlikeness associated with the lemniscate of Bernoulli and the left-half plane}, 
		Appl. Math. Comput. \textbf{218}(2012), no.~11, 6557--6565. 
		\bibitem{baba} Babalola,~K.O., \textit{On H3(1) Hankel determinants for some classes of univalent functions}. In: Dragomir,	S.S., Cho, J.Y. (eds.) Inequality theory and applications, \textbf{6}(2010), 1--7. Nova Science Publishers, New York.
		\bibitem{banga} Banga,~S, Kumar~S.S., \textit{The sharp bounds of the second and third Hankel determinants for the class $\mathcal{SL^*}$}, Mathematica Slovaca \textbf{70}(2020), no.~4, 849-862.
		\bibitem{cho} Cho,~N. E., Kumar,~V., \textit{Initial coefficients and fourth Hankel determinant for certain analytic functions}, Miskolc Mathematical Notes. \textbf{21}(2020), no.~2, 763--779.
		\bibitem{diens1957}
		Dienes,~P.:
		\textit{The Taylor series: an introduction to the theory of functions of a complex variable}, 
		Dover Publications, Inc., New York, 1957.
		\bibitem{ismail} Ismail, M.E.H., Merkes, E., Styer, D., \textit{A generalization of starlike functions}, Complex Variables, Theory and Application: An International Journal, \textbf{14}(1990), no.~1--4, 77--84.
		\bibitem{jack} Jackson, D.O., Fukuda, T., Dunn, O., Majors, E., \textit{On q-definite integrals}, In Quart. J. Pure Appl. Math.(1910)
		\bibitem{jan} Janowski, W., {\it Extremal problems for a family of functions with positive real part and for some related families}, Ann. Polon. Math. {\bf 23}(1970/1971), no. 2, 159--177.
		\bibitem{khanq} Khan, N., Shafiq, M., Darus, M., Khan, B., Ahmad, Q., \textit{Upper bound of the third Hankel determinant for a subclass of q-starlike functions associated with Lemniscate of Bernoulli}, J. Math. Inequal, \textbf{14}(2020), no. 1.
		\bibitem{kowalczyk2018}
		Kowalczyk,~B., Lecko,A., Lecko,~M., Sim,~Y. J., 
		\textit{The sharp bound of the third Hankel determinant for some classes of analytic functions}, 
		Bull. Korean Math. Soc. \textbf{55}(2018), no.~6, 1859--1868.
		\bibitem{kowalczyk22018}
		Kowalczyk,~B., Lecko,A., Lecko,~M., Sim,~Y. J., 
		\textit{The sharp bound for the Hankel determinant of the third kind for convex functions}, Bull. Aust. Math. Soc. \textbf{97}(2018), no.~3, 435--445. 
		\bibitem{kumar2013}
		Kumar,~S. S., Kumar,~V., Ravichandran,~V., Cho,~N. E., 
		\textit{Sufficient conditions for starlike functions associated with the lemniscate of Bernoulli},
		J. Inequal. Appl. \textbf{2013}, 2013:176, 13 pp.
		\bibitem{kwon2018}  
		Kwon,~O. S., Lecko,~A., Sim,~Y. J.,
		\textit{On the fourth coefficient of functions in the Carath\'{e}odory class},
		Comput. Methods Funct. Theory \textbf{18}(2018), no.~2, 307--314. 
		\bibitem{kwon2019}
		Kwon,~O. S., Lecko,~A., Sim,~Y. J.,
		\textit{The bound of the Hankel determinant of the third kind for starlike functions},  Bull. Malays. Math. Sci. Soc. \textbf{42}(2019), no.~2, 767--780. 
		\bibitem{lecko} Lecko, A., Sim, Y.J., ´Smiarowska, B., \textit{The sharp bound of the Hankel determinant of the third kind for
			starlike functions of order 1/2}, Complex Anal. Oper. Theory \textbf{13}(2019), no.~5, 2231–2238. 
		\bibitem{libera1982}
		Libera,~R. J., Z\l otkiwicz,~E. J., \textit{Early coefficients of the inverse of a regular convex function},
		Proc. Amer. Math. Soc. \textbf{85}(1982), no.~2, 225--230. 
		\bibitem{mahq} Mahmood, S., Jabeen, M., Malik, S.N., Srivastava, H.M., Manzoor, R., Riaz, S.M. \textit{Some coefficient inequalities of q-starlike functions associated with conic domain defined by q-derivative}, Journal of Function Spaces, 2018.
		\bibitem{men}Mendiratta, R., Nagpal, S., Ravichandran, V., {\it On a subclass of strongly starlike functions associated with exponential function}, Bull. Malays. Math. Sci. Soc. {\bf 38}(2015), no.~1, 365--386.
		\bibitem{noor1983}
		Noor,~K. I.,
		\textit{Hankel determinant problem for the class of functions with bounded boundary rotation}, 
		Rev. Roumaine Math. Pures Appl. \textbf{28}(1983), no.~8, 731--739. 
		\bibitem{pommerenke1966}
		Pommerenke,~C.: 
		\textit{On the coefficients and Hankel determinants of univalent functions},
		J. London Math. Soc. \textbf{1}(1966), no.~1, 111--122.
		\bibitem{shelly2015}
		Ravichandran,~V., Verma,~S.,	\textit{Bound for the fifth coefficient of certain starlike functions},
		C. R. Math. Acad. Sci. Paris \textbf{353}(2015), no.~6, 505--510.
		\bibitem{raza2013} 
		Raza,~M., Malik,~S. N.,
		\textit{Upper bound of the third Hankel determinant for a class of analytic functions related with lemniscate of Bernoulli},
		J. Inequal. Appl. \textbf{2013}(2013), no.~412, 8 pp. 
		\bibitem{robert}Robertson, M.S., {\it Certain classes of starlike functions}, Michigan Math. J. {\bf 32}(1985), no.~2, 135--140.
		\bibitem{ron} R\o nning, F., {\it Uniformly convex functions and a corresponding class of starlike functions}, Proc. Amer. Math. Soc. {\bf 118}(1993), no.~1, 189--196.
		\bibitem{sokol22009}
		Sok\'{o}\l, J.,  
		\textit{Radius problems in the class ${\mathscr{SL}}^*$}, 
		Appl. Math. Comput. \textbf{214}(2009), no.~2, 569--573.
		\bibitem{Sok}Sok\'{o}\l, J., Stankiewicz, J., {\it Radius of convexity of some subclasses of strongly starlike functions}, Zeszyty Nauk. Politech. Rzeszowskiej Mat. No. \textbf{19}(1996), 101--105.
		\bibitem{sriq} Srivastava, H.M., Deepak, B., \textit{Close-to-convexity of a certain family of q-Mittag-Leffler functions}, J. Nonlinear Var. Anal \textbf{1}(2017), no.~1, 61-69.
		\bibitem{zaprawa2016}
		Zaprawa,~P.,
		\textit{Second Hankel determinants for the class of typically real functions}, 
		Abstr. Appl. Anal. \textbf{2016}, Art. ID 3792367, 7 pp. 
		\bibitem{zaprawa2017} 
		Zaprawa,~P.,
		\textit{Third Hankel determinants for subclasses of univalent functions},
		Mediterr. J. Math. \textbf{14}(2017), no.~1, Art. 19, 10 pp.
		\bibitem{zaprawa21}Zaprawa,~P., Milutin, O., Tuneski, N., \textit{Third Hankel determinant for univalent starlike functions}, Revista de la Real Academia de Ciencias Exactas, Físicas y Naturales. Serie A. Matemáticas \textbf{115}(2021), no.~2, 1-6.
	\end{thebibliography}
\end{document}